\pdfoutput=1
\documentclass[twocolumn,a4paper]{article}

\usepackage[a4paper,margin=2.1cm]{geometry}
\usepackage[numbers,sort&compress]{natbib}
\usepackage{amsmath,amsthm,amssymb,mathtools}
\usepackage{enumitem}
\usepackage{graphicx}
\usepackage{hyperref}
\hypersetup{
  pdftitle={The exact region between Chatterjee's xi and Blomqvist's beta},
  pdfauthor={Jacob Israel Orenday Lares, Marcus Rockel},
  colorlinks=true,
  linkcolor={[rgb]{0.1,0.2,0.55}},
  citecolor={[rgb]{0.1,0.45,0.2}},
  urlcolor={[rgb]{0.1,0.2,0.55}}
}

\makeatletter
\renewcommand\section{\@startsection{section}{1}{\z@}{-3.5ex \@plus -1ex \@minus -.2ex}{2.3ex \@plus.2ex}{\normalfont\Large\bfseries\raggedright}}
\makeatother

\numberwithin{equation}{section}
\newtheorem{theorem}{Theorem}[section]
\newtheorem{proposition}[theorem]{Proposition}
\newtheorem{lemma}[theorem]{Lemma}
\newtheorem{corollary}[theorem]{Corollary}
\theoremstyle{remark}
\newtheorem{remark}[theorem]{Remark}

\newcommand{\CC}{\mathcal{C}}
\newcommand{\RR}{\mathcal{R}}
\newcommand{\AAA}{\mathcal{A}}
\newcommand{\R}{\mathbb{R}}
\newcommand{\1}{\mathbf{1}}
\newcommand{\sgn}{\operatorname{sgn}}
\newcommand{\de}{\,\mathrm{d}}
\newcommand{\PQD}{\mathrm{PQD}}
\newcommand{\NQD}{\mathrm{NQD}}
\newcommand{\SI}{\mathrm{SI}}
\newcommand{\SD}{\mathrm{SD}}
\newcommand{\RS}{\mathrm{RS}}
\newcommand{\TPtwo}{\mathrm{TP}_2}
\newcommand{\RRtwo}{\mathrm{RR}_2}

\begin{document}

\twocolumn[{%
  \begin{center}
    {\LARGE\bfseries The exact region between Chatterjee's \(\xi\)\\[2pt] and Blomqvist's \(\beta\)\par}
    \vspace{1.6ex}
    {\large Jacob Israel Orenday Lares\textsuperscript{1}\qquad Marcus Rockel\textsuperscript{2}\par}
    \vspace{1.0ex}
    {\small\itshape
    \textsuperscript{1}Corresponding author, \texttt{jacob.orenday@gmail.com}\\
    \textsuperscript{2}Department of Quantitative Finance, Institute for Economics, University of Freiburg,\\
    Rempartstr.~16, 79098 Freiburg, Germany\par}
    \vspace{1.0ex}
    {\normalsize\today\par}
    \vspace{2.0ex}
  \end{center}
  \begin{center}
  \begin{minipage}{0.86\textwidth}
    \begin{center}\textbf{Abstract}\end{center}
    \small
    We determine the exact region of Chatterjee's rank correlation \(\xi\) and Blomqvist's \(\beta\), that is, the set of all pairs \((\xi(C),\beta(C))\) where \(C\) ranges over all bivariate copulas.
    The region equals \(\{(x,y)\in[0,1]\times[-1,1]:|y|^3\le2x\}\), so every copula satisfies the sharp inequality \(|\beta(C)|^3\le2\xi(C)\).
    Its left boundary \(\xi=|\beta|^3/2\) is attained uniquely by an explicit family of copulas that perturb independence by a signed tent function centered at the median, and its right boundary \(\xi=1\) by deterministic interval-exchange copulas.
    The proof reduces the problem to a one-dimensional minimization of an \(L^2\)-norm over \(1\)-Lipschitz functions.
    We further show that the region does not shrink when restricted to radially symmetric copulas and that positively quadrant dependent copulas attain exactly its upper half.
    For copulas under which the second variable is stochastically increasing in the first, the exact region is \(\{(x,y)\in[0,1]^2:y^3\le2x,\ 3(1-y)^2\le4(1-x)\}\), and its right boundary is attained uniquely by explicit singular two-branch copulas.
    In particular, under stochastic increasingness \(\beta=0\) is compatible with every value \(\xi\in[0,\tfrac14]\), whereas \(\xi>\tfrac14\) forces \(\beta>0\).
    \medskip

  \end{minipage}
  \end{center}
  \vspace{2.4ex}
}]

\section{Introduction}\label{sec:intro}

Dependence between two random variables can rarely be summarized by a single number, and in applications several rank-based coefficients are therefore often computed side by side, each condensing a different aspect of the dependence structure into one value.
This raises a basic question: which combinations of values are possible at all?
The answer, the \emph{exact region} of a pair of coefficients, describes precisely how much information one coefficient carries about the other, and its boundary yields sharp inequalities between them.
Since the region of Kendall's \(\tau\) and Spearman's \(\rho\) was determined in \cite{schreyer2017exact}, exact regions have been obtained for most pairs of classical measures of concordance, see \cite{kokolbukovsek2019relation,bukovvsek2021exact,bukovvsek2022exact,bukovvsek2023exact,bukovvsek2024exact} and the references therein, and more recently for triples of them \cite{bukovvsek2025exact,kokolbukovsek2026footrule,orenday2026taufootrulebeta}.

Chatterjee's rank correlation \(\xi\) \cite{chatterjee2020} is a coefficient of a different kind.
It does not ask whether large values of one variable tend to occur together with large values of the other.
It asks how close the second variable is to being a function of the first: \(\xi\) vanishes exactly under independence and equals one exactly when the second variable is almost surely a measurable function of the first, see \cite[Thm.~1.1]{chatterjee2020}.
The population version appeared earlier, as a copula-based measure of regression dependence, in \cite{dette2013copula}.
What made \(\xi\) popular was the simple and consistent rank estimator of \cite{chatterjee2020}, which has since triggered a large body of work on its power, its extensions, and its conditional versions, see, e.g., \cite{chatterjee2021,shi2022power,auddy2021exact,deb2020b,Ansari-LFT-2025}.
Because \(\xi\) measures the strength of a functional relationship regardless of whether it is increasing or decreasing, no simple comparison with measures of concordance exists.
Exact regions involving \(\xi\) are correspondingly recent: with Spearman's \(\rho\) \cite{ansari2026exact}, with Spearman's footrule \cite{rockel2026exact}, with Blest's rank correlation \cite{rockel2026blest}, and with \(\tau\), \(\rho\), and the footrule within the class of lower semilinear copulas \cite{fuchs2025exact}.

Blomqvist's \(\beta\) \cite{blomqvist1950measure} sits at the other end of the spectrum, as arguably the simplest measure of concordance.
All it does is compare how often the two variables lie on the same side of their medians with how often they lie on opposite sides.
The entire dependence structure thus enters through a single \(2\times2\) contingency table.
How \(\beta\) relates to the other classical measures of concordance was analyzed in \cite{kokolbukovsek2019relation,bukovvsek2021exact}.
The pair \((\xi,\beta)\) therefore confronts a global measure of functional dependence with the most local measure of concordance.
The present paper determines its exact region, both over all copulas and over the class of stochastically increasing copulas.

Some notation is needed to state the results.
Let \(\CC\) denote the class of bivariate copulas.
If \((X,Y)\) has continuous marginal distribution functions \(F_X\) and \(F_Y\), then Sklar's theorem \cite{sklar1959fonctions} provides a unique \(C\in\CC\) with \(F_{X,Y}(x,y)=C(F_X(x),F_Y(y))\), and every coefficient considered here is a functional of \(C\), see, e.g., \cite{nelsen2006introduction,durante2016principles}.
In terms of the copula, Chatterjee's rank correlation reads
\begin{equation}\label{eq:xi-def}
\xi(C)=6\int_0^1\!\!\int_0^1\bigl(\partial_1C(u,v)\bigr)^2\de u\de v-2,
\end{equation}
where \(\partial_1C(u,v)\) denotes the conditional distribution function of the second coordinate given the first, see Section~\ref{sec:preliminaries}, and Blomqvist's beta reads
\begin{equation}\label{eq:beta-def}
\beta(C)=4C\bigl(\tfrac12,\tfrac12\bigr)-1.
\end{equation}
For a subclass \(\AAA\subseteq\CC\), we write
\[
\RR^{\AAA}_{\xi,\beta}\coloneqq\{(\xi(C),\beta(C)):C\in\AAA\}
\]
and \(\RR_{\xi,\beta}\coloneqq\RR^{\CC}_{\xi,\beta}\).
Figure~\ref{fig:region} illustrates our first main result.

\begin{theorem}[Exact \(\xi\)--\(\beta\) region]\label{thm:main}
It holds that
\[
\RR_{\xi,\beta}=\bigl\{(x,y)\in[0,1]\times[-1,1]:|y|^3\le2x\bigr\}.
\]
In particular, every \(C\in\CC\) satisfies
\begin{equation}\label{eq:main-ineq}
|\beta(C)|^3\le2\,\xi(C),
\end{equation}
with equality if and only if \(C=L_{\beta(C)}\) for the tent copulas \(L_b\) of Section~\ref{sec:tent}.
Every point of the right boundary \(\{1\}\times[-1,1]\) is attained by a deterministic copula \(D_b\) of Proposition~\ref{prop:right-boundary}.
\end{theorem}

\begin{figure}[t!]
\centering
\includegraphics[width=\linewidth]{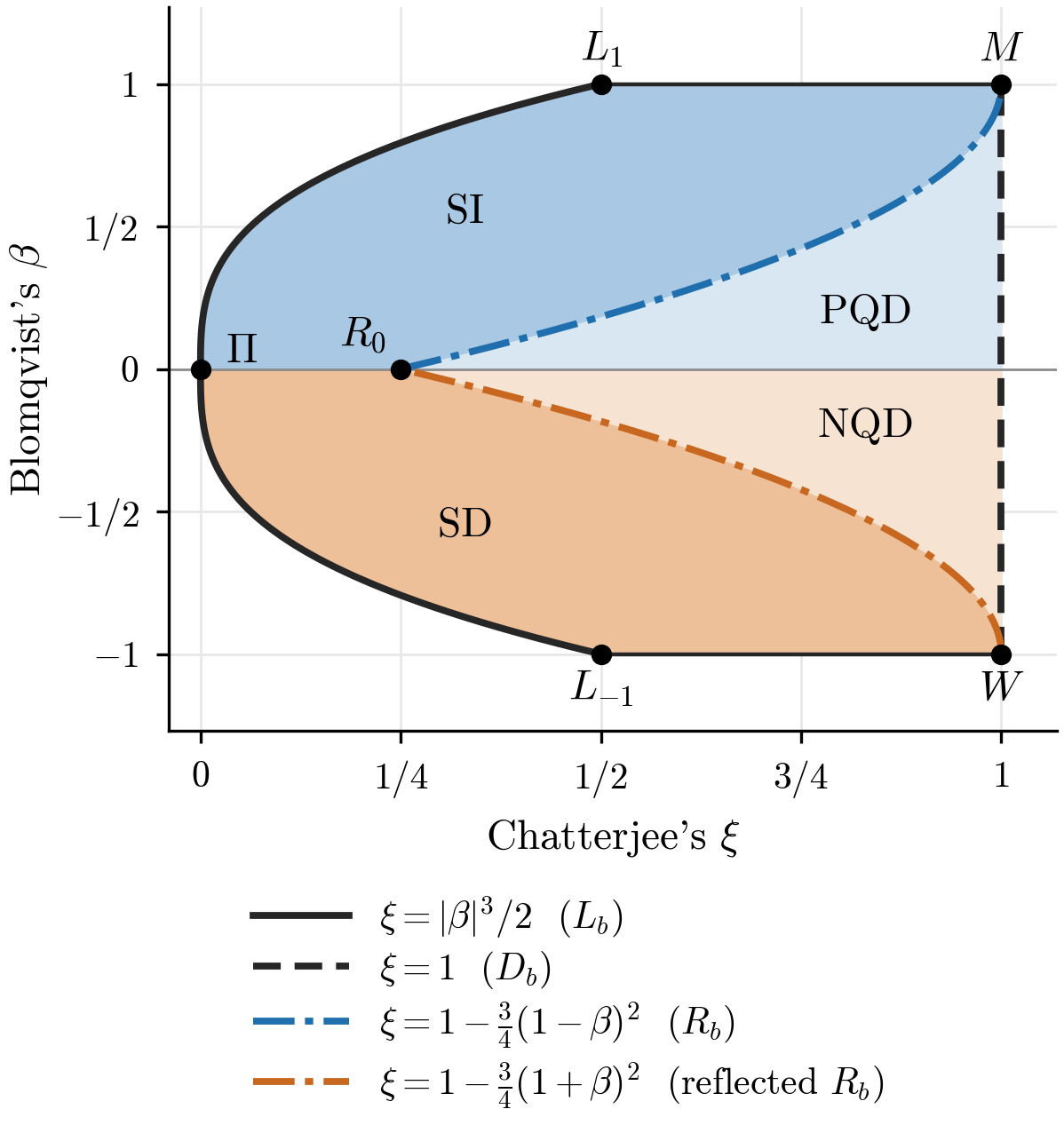}
\caption{The exact region \(\RR_{\xi,\beta}\) of Theorem~\ref{thm:main} and its subclasses, see Theorem~\ref{thm:si} and Corollary~\ref{cor:rs-pqd}.
The full region is bounded by the cubic \(\xi=|\beta|^3/2\) (solid), attained uniquely by the tent copulas \(L_b\), and by the line \(\xi=1\) (dashed), attained by the deterministic copulas \(D_b\).
Mixtures of \(L_{\pm1}\) and \(D_{\pm1}\) fill the horizontal edges at \(\beta=\pm1\).
Radially symmetric copulas attain all of \(\RR_{\xi,\beta}\), \(\PQD\) copulas exactly its upper half (both shades of blue), and \(\NQD\) copulas exactly its lower half (both shades of orange).
The \(\SI\) region is the dark blue set between the cubic and the curve \(\xi=1-\tfrac34(1-\beta)^2\) (dash-dotted, two-branch copulas \(R_b\)).
Its mirror image (dark orange) is the \(\SD\) region, whose right boundary is attained by the copulas \(\check R_{|\beta|}\).
At \(\beta=0\), the \(\SI\) region degenerates to the segment \(0\le\xi\le\tfrac14\) between \(\Pi\) and \(R_0\).
Dots mark \(\Pi\), the Fr\'echet--Hoeffding bounds \(M\) and \(W\), the tent copulas \(L_{\pm1}\), and \(R_0\).
Every curve is drawn from a closed-form expression.}
\label{fig:region}
\end{figure}

Inequality~\eqref{eq:main-ineq} restricts the pair \((\xi,\beta)\) in one direction only.
A small value of \(\xi\) pins \(\beta\) near zero, since \(|\beta(C)|\le(2\xi(C))^{1/3}\).
A given value of \(\beta\), in contrast, bounds \(\xi\) only from below.
Even the extreme values \(\beta=\pm1\) leave room for every \(\xi\in[\tfrac12,1]\), and the lower endpoint \(\xi=\tfrac12\) is attained by the tent copulas \(L_{\pm1}\), which spread their mass uniformly over two opposite median quadrants, see Figure~\ref{fig:tent}.

A short variational argument proves Theorem~\ref{thm:main}.
Conditioning on the half of the unit interval that contains the first coordinate produces two conditional distribution functions \(F^+\) and \(F^-\) of the second coordinate whose average is the uniform distribution function.
The value of \(\beta\) fixes the displacement \(g(v)\coloneqq F^+(v)-v\) at the median, \(g(\tfrac12)=\beta/2\), while monotonicity of \(F^\pm\) forces \(g\) to be \(1\)-Lipschitz.
Jensen's inequality then gives
\begin{equation}\label{eq:key}
\xi(C)\ge6\int_0^1g(v)^2\de v.
\end{equation}
Among all \(1\)-Lipschitz functions with value \(\beta/2\) at \(\tfrac12\), exactly one minimizes the right-hand side: the tent \(v\mapsto\sgn(\beta)(|\beta|/2-|v-\tfrac12|)_+\), with minimal value \(|\beta|^3/2\).

Many statistical models rest on positive dependence assumptions.
A standard one is \emph{stochastic increasingness}: the conditional distribution of the second variable, given that the first equals \(u\), is stochastically increasing in \(u\), see \cite{lehmann1966concepts} and Section~\ref{sec:preliminaries}.
Some measures of concordance dominate \(\xi\) under this assumption, for instance \(\xi\le\tau\) \cite[Thm.~1.1]{rockel2026kendall}.
Blomqvist's \(\beta\) is not among them.
It does, however, now bound \(\xi\) from above as well.
Denote by \(\CC_{\SI}\), \(\CC_{\SD}\), and \(\CC_{\RS}\) the classes of stochastically increasing, stochastically decreasing, and radially symmetric copulas, and write \(\RR^{\SI}_{\xi,\beta}\coloneqq\RR^{\CC_{\SI}}_{\xi,\beta}\), \(\RR^{\SI,\RS}_{\xi,\beta}\coloneqq\RR^{\CC_{\SI}\cap\CC_{\RS}}_{\xi,\beta}\), and similarly for other classes.

\begin{theorem}[Exact region under stochastic increasingness]\label{thm:si}
It holds that
\begin{multline*}
\RR^{\SI}_{\xi,\beta}=\RR^{\SI,\RS}_{\xi,\beta}=\bigl\{(x,y)\in[0,1]^2:\\
y^3\le2x,\ 3(1-y)^2\le4(1-x)\bigr\}.
\end{multline*}
In particular, every \(C\in\CC_{\SI}\) satisfies
\begin{equation}\label{eq:si-ineq}
\xi(C)\le1-\tfrac34\bigl(1-\beta(C)\bigr)^2,
\end{equation}
with equality if and only if \(C=R_{\beta(C)}\) for the two-branch copulas \(R_b\) of Proposition~\ref{prop:Rb}.
Reflection yields \(\RR^{\SD}_{\xi,\beta}=\RR^{\SD,\RS}_{\xi,\beta}=\{(x,-y):(x,y)\in\RR^{\SI}_{\xi,\beta}\}\).
\end{theorem}

Inequality~\eqref{eq:si-ineq} can be rewritten as
\[
\beta(C)\ge1-2\sqrt{\bigl(1-\xi(C)\bigr)/3},\qquad C\in\CC_{\SI}.
\]
For stochastically increasing copulas, \(\xi>\tfrac14\) therefore forces \(\beta>0\), and \(\xi\) close to one forces \(\beta\) close to one.
In the other direction, \(\beta=0\) remains compatible with every \(\xi\in[0,\tfrac14]\), see Figure~\ref{fig:region}.
Spearman's \(\rho\) behaves differently, since for positively quadrant dependent copulas, and in particular for stochastically increasing ones, it vanishes only under independence.
The copulas \(R_b\) on the new boundary are singular.
For \(0<b<1\), \(R_b\) is an ordinal sum of two comonotonic blocks and one scaled copy of \(R_0\), the copula of \((U,\tfrac12(U+\varepsilon))\) for independent \(U\sim\mathrm{U}(0,1)\) and a fair coin \(\varepsilon\in\{0,1\}\), see Figure~\ref{fig:singular} and Remark~\ref{rem:R0}.
The proof of Theorem~\ref{thm:si} makes two moves.
First, it replaces a given copula by the largest copula with the same median section \(C(\cdot,\tfrac12)\), which is again stochastically increasing and has a value of \(\xi\) at least as large.
Second, it reduces the problem to a sharp inequality for convex functions (Lemmas~\ref{lem:completion} and~\ref{lem:convex}).
Corollary~\ref{cor:rs-pqd} completes the picture: radial symmetry imposes no restriction on the region of Theorem~\ref{thm:main}, and positively quadrant dependent copulas attain exactly its upper half.

The paper is organized as follows.
Section~\ref{sec:preliminaries} collects what we need about copulas, Markov kernels, and dependence concepts, including an interpolation lemma behind every attainment statement.
Section~\ref{sec:tent} introduces the tent copulas \(L_b\) and records their properties.
Section~\ref{sec:proof-main} proves the sharp inequality~\eqref{eq:main-ineq}, constructs the right-boundary copulas \(D_b\), and completes the proof of Theorem~\ref{thm:main}.
Section~\ref{sec:si} treats stochastically increasing copulas and proves Theorem~\ref{thm:si}, and Section~\ref{sec:conclusion} discusses open problems.

\section{Preliminaries}\label{sec:preliminaries}

A copula is the distribution function of a random vector \((U,V)\) whose components are uniformly distributed on \([0,1]\).
Equivalently, it is a function \(C\colon[0,1]^2\to[0,1]\) that is grounded, \(2\)-increasing, and has uniform margins, see \cite[Ch.~2]{nelsen2006introduction}.
We write \((U,V)\sim C\) and denote by \(\Pi(u,v)\coloneqq uv\), \(M(u,v)\coloneqq\min\{u,v\}\), and \(W(u,v)\coloneqq\max\{u+v-1,0\}\) the independence copula and the Fr\'echet--Hoeffding bounds.
Every copula satisfies \(W\le C\le M\) and is \(1\)-Lipschitz in each argument, see \cite[Sec.~2.2]{nelsen2006introduction}, and the class \(\CC\) is convex.
If \((U,V)\sim C\), then \((V,U)\), \((1-U,1-V)\), and \((U,1-V)\) have the copulas
\begin{align*}
C^\top(u,v)&\coloneqq C(v,u),\\
\widehat C(u,v)&\coloneqq u+v-1+C(1-u,1-v),\\
\check C(u,v)&\coloneqq u-C(u,1-v),
\end{align*}
respectively.
We call \(C\) \emph{exchangeable} if \(C=C^\top\) and \emph{radially symmetric} if \(C=\widehat C\).

\emph{Markov kernels.}
Every \(C\in\CC\) admits a Markov kernel \(K_C\) from \([0,1]\) to \([0,1]\), a regular conditional distribution of \(V\) given \(U\), such that, for every \(v\in[0,1]\), the partial derivative \(\partial_1C(u,v)\) exists and equals \(K_C(u,[0,v])\) for Lebesgue-a.e.\ \(u\), see \cite{darsow1992copulas,durante2016principles}.
We fix such a kernel and write \(\partial_1C(u,v)\coloneqq K_C(u,[0,v])\) for all \((u,v)\).
Then \(v\mapsto\partial_1C(u,v)\) is a distribution function on \([0,1]\) for every \(u\), and
\begin{equation}\label{eq:kernel}
C(u,v)=\int_0^u\partial_1C(t,v)\de t,
\end{equation}
in particular \(\int_0^1\partial_1C(t,v)\de t=v\) for all \(v\).
The construction also runs backwards.
Every Markov kernel \(K\) with \(\int_0^1K(t,[0,v])\de t=v\) for all \(v\) defines a copula through~\eqref{eq:kernel}, namely the distribution function of \((U,V)\) with \(U\sim\mathrm{U}(0,1)\) and conditional law \(K(u,\cdot)\) of \(V\) given \(U=u\).
All constructions below are of this form.

\emph{Dependence measures.}
Chatterjee's rank correlation~\eqref{eq:xi-def} does not depend on the choice of the kernel, and it satisfies \(0\le\xi(C)\le1\).
The lower endpoint characterizes independence, \(\xi(C)=0\) if and only if \(C=\Pi\).
The upper endpoint characterizes \emph{complete dependence}: \(\xi(C)=1\) if and only if \(V=f(U)\) almost surely for some Borel function \(f\), which is then necessarily Lebesgue measure preserving, see \cite[Thm.~1.1]{chatterjee2020} and \cite{dette2013copula}.
Since \(\int_0^1\!\int_0^1\partial_1C(u,v)\de u\de v=\tfrac12\) by~\eqref{eq:kernel}, we may rewrite~\eqref{eq:xi-def} as
\begin{equation}\label{eq:xi-deficit}
1-\xi(C)=6\int_0^1\!\!\int_0^1h(1-h)\de u\de v,
\end{equation}
with \(h=\partial_1C(u,v)\).
Blomqvist's beta~\eqref{eq:beta-def} takes values in \([-1,1]\) by the Fr\'echet--Hoeffding bounds and is affine on \(\CC\).
We also use Spearman's rho and Kendall's tau,
\begin{align*}
\rho(C)&=12\int_0^1\!\!\int_0^1C(u,v)\de u\de v-3,\\
\tau(C)&=4\int_0^1\!\!\int_0^1C(u,v)\de C(u,v)-1,
\end{align*}
see \cite[Sec.~5.1]{nelsen2006introduction}, that is, \(\rho(C)=12\,\mathbb E[UV]-3\) and \(\tau(C)=4\,\mathbb E[C(U,V)]-1\) for \((U,V)\sim C\).
Since \(\partial_1\check C(u,v)=1-\partial_1C(u,1-v)\) for every \(v\) and a.e.\ \(u\), a short computation based on~\eqref{eq:kernel} shows
\begin{equation}\label{eq:flip}
\xi(\check C)=\xi(C),\qquad\beta(\check C)=-\beta(C).
\end{equation}

\emph{Dependence concepts.}
A copula \(C\) is \emph{positively quadrant dependent} (\(\PQD\)) if \(C\ge\Pi\) and \emph{negatively quadrant dependent} (\(\NQD\)) if \(C\le\Pi\).
It is \emph{stochastically increasing} (\(\SI\)), more precisely, \(V\) is stochastically increasing in \(U\), if the kernel can be chosen such that \(u\mapsto\partial_1C(u,v)\) is nonincreasing for every \(v\).
It is \emph{stochastically decreasing} (\(\SD\)) if \(u\mapsto\partial_1C(u,v)\) can be chosen nondecreasing for every \(v\).
Equivalently, \(C(\cdot,v)\) is concave, respectively convex, for every \(v\), and hence \(C^\top\) is \(\SI\) if and only if \(v\mapsto C(u,v)\) is concave for every \(u\).
Both notions are directional, so \(C\) may be \(\SI\) while \(C^\top\) is not.
Every \(\SI\) copula is \(\PQD\), and every \(\SD\) copula is \(\NQD\), see \cite{lehmann1966concepts} and \cite[Sec.~5.2]{nelsen2006introduction}.
For an absolutely continuous copula with density \(c\), \emph{total positivity of order two} (\(\TPtwo\)) means \(c(u_1,v_1)c(u_2,v_2)\ge c(u_1,v_2)c(u_2,v_1)\) for a.e.\ \(u_1<u_2\) and \(v_1<v_2\), and the \emph{reverse rule of order two} (\(\RRtwo\)) means the reverse inequality, see \cite{karlin1968total,karlinrinott1980classes,karlinrinott1980reverse}.
Besides the classes \(\CC_{\SI}\), \(\CC_{\SD}\), and \(\CC_{\RS}\) from the introduction, we denote by \(\CC_{\PQD}\) and \(\CC_{\NQD}\) the classes of \(\PQD\) and \(\NQD\) copulas.
All five classes are convex, since a mixture of copulas has the corresponding mixture of kernels.
The involution \(C\mapsto\check C\) interchanges \(\CC_{\PQD}\) with \(\CC_{\NQD}\) and \(\CC_{\SI}\) with \(\CC_{\SD}\), and it maps \(\CC_{\RS}\) onto itself.
Hence, by~\eqref{eq:flip}, for each of these classes and each intersection \(\AAA\) of them, with \(\check\AAA\coloneqq\{\check C:C\in\AAA\}\),
\begin{equation}\label{eq:reflect-region}
\RR^{\check\AAA}_{\xi,\beta}=\bigl\{(x,-y):(x,y)\in\RR^{\AAA}_{\xi,\beta}\bigr\}.
\end{equation}

Every attainment statement in this paper rests on the following interpolation lemma.

\begin{lemma}[Interpolation]\label{lem:interpolation}
Let \(\AAA\subseteq\CC\) be convex, and let \(C_0,C_1\in\AAA\) with \(\beta(C_0)=\beta(C_1)=b\).
Then, for every \(x\) between \(\xi(C_0)\) and \(\xi(C_1)\), there is \(C\in\AAA\) with \(\xi(C)=x\) and \(\beta(C)=b\).
\end{lemma}

\begin{proof}
For \(\lambda\in[0,1]\), the copula \(C_\lambda\coloneqq(1-\lambda)C_0+\lambda C_1\) lies in \(\AAA\) and satisfies \(\beta(C_\lambda)=b\), since \(\beta\) is affine.
Its kernel is \((1-\lambda)\partial_1C_0+\lambda\partial_1C_1\), so by~\eqref{eq:xi-def} the map \(\lambda\mapsto\xi(C_\lambda)\) is a polynomial of degree at most two.
The intermediate value theorem finishes the proof.
\end{proof}

\section{Tent copulas}\label{sec:tent}

This section introduces the copulas on the left boundary of \(\RR_{\xi,\beta}\).
Their construction follows the variational argument sketched in the introduction.
An extremal copula should have a kernel that is constant on each half of the first coordinate, and the displacement between the two resulting conditional distribution functions should be as small as the constraint \(\beta=b\) allows, which leads to a tent centered at the median.
Proposition~\ref{prop:Lb} shows that this recipe produces a copula with \(\beta=b\) and \(\xi=|b|^3/2\).
Proposition~\ref{prop:Lb-properties} collects further structural properties, among them the symmetry and stochastic monotonicity on which the subclass results of Sections~\ref{sec:proof-main} and~\ref{sec:si} rely.

Fix \(b\in[-1,1]\), put \(\alpha_b\coloneqq(1-|b|)/2\), and define the signed tent
\begin{equation}\label{eq:tent}
g_b(v)\coloneqq\sgn(b)\bigl(\tfrac{|b|}2-\bigl|v-\tfrac12\bigr|\bigr)_+,\qquad v\in[0,1],
\end{equation}
with \(\sgn(0)\coloneqq0\).
The function \(g_b\) vanishes outside \([\alpha_b,1-\alpha_b]\), is \(1\)-Lipschitz, and satisfies \(g_b(\tfrac12)=b/2\), \(g_b(1-v)=g_b(v)\), \(g_{-b}=-g_b\), and \(\int_0^1g_b(v)\de v=b|b|/4\).
Its derivative is
\[
g_b'=\sgn(b)\bigl(\1_{(\alpha_b,1/2)}-\1_{(1/2,1-\alpha_b)}\bigr)\quad\text{a.e.}
\]
With \(\ell(u)\coloneqq\min\{u,1-u\}\), and \(\sigma(u)\coloneqq1\) for \(u<\tfrac12\) and \(\sigma(u)\coloneqq-1\) for \(u\ge\tfrac12\), we define the \emph{tent copula}
\begin{equation}\label{eq:Lb}
L_b(u,v)\coloneqq uv+\ell(u)\,g_b(v),\qquad(u,v)\in[0,1]^2.
\end{equation}
In other words, \(L_b=\Pi+\ell\otimes g_b\) with \((\ell\otimes g_b)(u,v)\coloneqq\ell(u)g_b(v)\), a perturbation of independence by the product of two tents.
It equals \(u(v+g_b(v))\) on the strip \(u\le\tfrac12\) and \(uv+(1-u)g_b(v)\) on the strip \(u\ge\tfrac12\).
Figure~\ref{fig:tent} displays its density.

\begin{proposition}[Tent copulas]\label{prop:Lb}
For every \(b\in[-1,1]\), the function \(L_b\) is an absolutely continuous copula with density
\begin{equation}\label{eq:Lb-density}
c_b(u,v)=1+\sigma(u)\,g_b'(v)\in\{0,1,2\}
\end{equation}
for a.e.\ \((u,v)\), and with Markov kernel
\begin{equation}\label{eq:Lb-kernel}
\partial_1L_b(u,v)=v+\sigma(u)\,g_b(v).
\end{equation}
It satisfies \(\beta(L_b)=b\) and \(\xi(L_b)=|b|^3/2\).
\end{proposition}

\begin{proof}
The functions \(\ell\) and \(g_b\) are Lipschitz with \(\ell(0)=g_b(0)=0\) and \(\ell'=\sigma\) a.e., so~\eqref{eq:Lb} gives \(L_b(u,v)=\int_0^u\int_0^vc_b(s,t)\de t\de s\) with \(c_b\) as in~\eqref{eq:Lb-density}.
Since \(\sigma(u)g_b'(v)\in\{-1,0,1\}\) a.e., we have \(c_b\ge0\), and \(\ell(1)=g_b(1)=0\) gives \(L_b(u,1)=u\) and \(L_b(1,v)=v\).
Hence \(L_b\) is a copula with density \(c_b\).
Differentiating in \(u\) gives~\eqref{eq:Lb-kernel}, and since the slopes of \(g_b\) lie in \([-1,1]\), the functions \(v\mapsto v\pm g_b(v)\) are distribution functions on \([0,1]\).

From \(\ell(\tfrac12)=\tfrac12\) and \(g_b(\tfrac12)=b/2\) we get \(L_b(\tfrac12,\tfrac12)=(1+b)/4\), so \(\beta(L_b)=b\).
By~\eqref{eq:Lb-kernel} and \(\int_0^1\sigma(u)\de u=0\),
\begin{align*}
\int_0^1\!\!\int_0^1\bigl(\partial_1L_b(u,v)\bigr)^2\de u\de v
&=\int_0^1\bigl(v^2+g_b(v)^2\bigr)\de v\\
&=\tfrac13+\tfrac{|b|^3}{12},
\end{align*}
since \(g_b\) is a tent of height and half-width \(|b|/2\).
Now~\eqref{eq:xi-def} yields \(\xi(L_b)=|b|^3/2\).
\end{proof}

\begin{figure*}[t!]
\centering
\includegraphics[width=\textwidth]{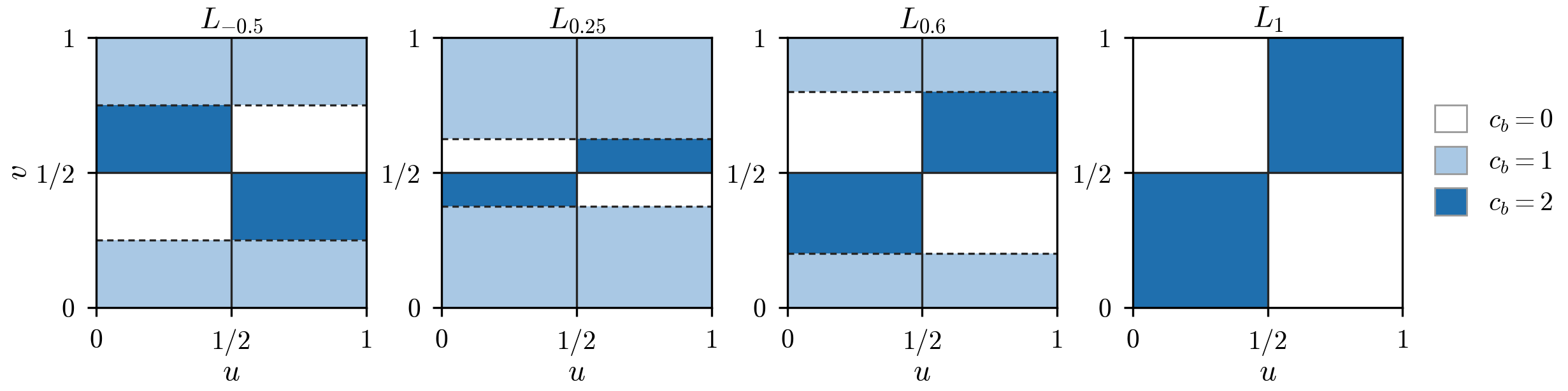}
\caption{Densities \(c_b\) of the tent copulas \(L_b\) for \(b\in\{-0.5,0.25,0.6,1\}\), see~\eqref{eq:Lb-density}.
Each density is constant on the rectangles cut out by the median \(u=\tfrac12\) (solid vertical line) and by the horizontal lines \(v=\alpha_b\), \(v=\tfrac12\), \(v=1-\alpha_b\), where \(\alpha_b=(1-|b|)/2\).
It takes the value \(2\) on the dark rectangles and \(0\) on the white ones, so the tent \(g_b\) shifts the conditional distribution functions \(v\mapsto v\pm g_b(v)\) on the left and right strip against each other.
At \(b=1\) the mass spreads uniformly over the two diagonal median quadrants.
Reflecting \(L_b\) in \(v\) yields \(L_{-b}\).}
\label{fig:tent}
\end{figure*}

The tent copulas share several structural properties, and these matter whenever \(L_b\) serves as a boundary copula of a subclass.

\begin{proposition}[Further properties of \(L_b\)]\label{prop:Lb-properties}
Let \(b\in[-1,1]\).
\begin{enumerate}[label=\textup{(\roman*)},leftmargin=*]
\item \emph{Symmetries.} \(\check L_b=L_{-b}\) and \(\widehat L_b=L_b\).
Exchangeability holds if and only if \(b\in\{-1,0,1\}\).
Here, \(L_0=\Pi\), \(L_1\) is the ordinal sum of \(\Pi\) and \(\Pi\) with respect to \(\{[0,\tfrac12],[\tfrac12,1]\}\), and \(L_{-1}=\check L_1\).
\item \emph{Stochastic monotonicity.} \(L_b\) is \(\SI\) for \(b\ge0\) and \(\SD\) for \(b\le0\).
Consequently, \(L_b\) is \(\PQD\) if and only if \(b\ge0\), and \(\NQD\) if and only if \(b\le0\).
\item \emph{Reverse direction.} \(L_b^\top\) is \(\SI\) if and only if \(b\in\{0,1\}\), and \(\SD\) if and only if \(b\in\{-1,0\}\).
\item \emph{Total positivity.} \(c_b\) is \(\TPtwo\) if and only if \(b\in\{0,1\}\), and \(\RRtwo\) if and only if \(b\in\{-1,0\}\).
\item \emph{Concordance.} \(\rho(L_b)=3b|b|/4\) and \(\tau(L_b)=b|b|/2\).
\end{enumerate}
\end{proposition}

\begin{proof}
We use \(g_b(1-v)=g_b(v)\), \(g_{-b}=-g_b\), and \(\ell(1-u)=\ell(u)\).

(i) Substituting into~\eqref{eq:Lb} gives
\begin{align*}
\check L_b(u,v)&=u-u(1-v)-\ell(u)g_b(v)=L_{-b}(u,v),\\
\widehat L_b(u,v)&=u+v-1+(1-u)(1-v)+\ell(u)g_b(v),
\end{align*}
and the latter equals \(L_b(u,v)\).
For \(b\in\{-1,1\}\) we have \(g_b=b\,\ell\), so \(L_{\pm1}(u,v)=uv\pm\ell(u)\ell(v)\) is exchangeable, and \(L_0=\Pi\).
By~\eqref{eq:Lb-density}, \(c_1=2\) on \([0,\tfrac12]^2\cup[\tfrac12,1]^2\) and \(c_1=0\) elsewhere, which is the ordinal sum of two copies of \(\Pi\).
For \(0<|b|<1\) and \(u\in(0,\alpha_b)\), we have \(g_b(u)=0\), so
\[
L_b\bigl(u,\tfrac12\bigr)=\tfrac u2+\tfrac b2\,\ell(u)\ne\tfrac u2=L_b\bigl(\tfrac12,u\bigr),
\]
and \(L_b\) is not exchangeable.

(ii) By~\eqref{eq:Lb-kernel}, \(\partial_1L_b(u,v)=v+\sigma(u)g_b(v)\).
For \(b\ge0\) we have \(g_b\ge0\), so this is nonincreasing in \(u\), and \(L_b\) is \(\SI\), hence \(\PQD\).
For \(b\le0\) it is nondecreasing in \(u\), so \(L_b\) is \(\SD\), hence \(\NQD\).
Since \(L_b-\Pi=\ell\otimes g_b\) has the sign of \(b\) and does not vanish for \(b\ne0\), \(L_b\) is neither \(\PQD\) for \(b<0\) nor \(\NQD\) for \(b>0\).

(iii) For \(b=0\) the claim is trivial.
For \(b=\pm1\), \(L_b^\top=L_b\) by~(i), so \(L_1^\top\) is \(\SI\) by~(ii), and it is not \(\SD\), since \(\SD\) copulas are \(\NQD\) and \(L_1-\Pi=\ell\otimes\ell\) is positive on \((0,1)^2\).
The case \(b=-1\) is symmetric.
Let \(0<|b|<1\).
Applying the kernel representation to \(L_b^\top\), the conditional distribution function of \(U\) given \(V=v\) is \(u\mapsto\partial_2L_b(u,v)=u+\ell(u)g_b'(v)\) for a.e.\ \(v\), and \(L_b^\top\) is \(\SI\) (\(\SD\)) if and only if a version of it is nonincreasing (nondecreasing) in \(v\) for every \(u\).
At \(u=\tfrac12\) it equals \(\tfrac12+\tfrac12g_b'(v)\).
For \(0<b<1\), this takes the values \(\tfrac12,1,0,\tfrac12\) on the four intervals \((0,\alpha_b)\), \((\alpha_b,\tfrac12)\), \((\tfrac12,1-\alpha_b)\), and \((1-\alpha_b,1)\), and for \(-1<b<0\) the values \(\tfrac12,0,1,\tfrac12\).
Both sequences are non-monotone on sets of positive measure, so \(L_b^\top\) is neither \(\SI\) nor \(\SD\).

(iv) For \(b=0\), \(c_0\equiv1\) is both \(\TPtwo\) and \(\RRtwo\).
For \(b=1\), \(c_1\) is \(2\) on the two diagonal and \(0\) on the two anti-diagonal median quadrants, and a direct case check shows that \(c_1\) is \(\TPtwo\).
It is not \(\RRtwo\), since \(u_1<\tfrac12<u_2\) and \(v_1<\tfrac12<v_2\) give
\[
c_1(u_1,v_1)c_1(u_2,v_2)=4>0=c_1(u_1,v_2)c_1(u_2,v_1).
\]
By~(i), \(c_{-1}(u,v)=c_1(u,1-v)\), so \(c_{-1}\) is \(\RRtwo\) but not \(\TPtwo\).
Let \(0<|b|<1\), and let always \(u_1<\tfrac12<u_2\).
For \(0<b<1\), the density equals \(1,2,0,1\) on the four intervals of~(iii) in \(v\) if \(u<\tfrac12\), and \(1,0,2,1\) if \(u>\tfrac12\).
Taking \(v_1<\alpha_b<v_2<\tfrac12\) gives \(c_b(u_1,v_1)c_b(u_2,v_2)=0<2=c_b(u_1,v_2)c_b(u_2,v_1)\), so \(c_b\) is not \(\TPtwo\), and taking \(\alpha_b<v_1<\tfrac12<v_2<1-\alpha_b\) gives \(4>0\), so \(c_b\) is not \(\RRtwo\).
For \(-1<b<0\), the values \(0\) and \(2\) are interchanged.
Now \(v_1<\alpha_b<v_2<\tfrac12\) gives \(2>0\), so \(c_b\) is not \(\RRtwo\), and \(\alpha_b<v_1<\tfrac12<v_2<1-\alpha_b\) gives \(0<4\), so \(c_b\) is not \(\TPtwo\).

(v) Using \(\int_0^1\ell(u)\de u=\tfrac14\) and \(\int_0^1g_b(v)\de v=b|b|/4\), we get
\[
\rho(L_b)=12\Bigl(\frac14+\frac14\cdot\frac{b|b|}4\Bigr)-3=\frac{3b|b|}4.
\]
For Kendall's tau, we expand \(L_bc_b\) by~\eqref{eq:Lb} and~\eqref{eq:Lb-density}.
With \(\int_0^1u\,\sigma(u)\de u=-\tfrac14\), \(\int_0^1\ell(u)\sigma(u)\de u=0\), and \(\int_0^1v\,g_b'(v)\de v=-b|b|/4\), this gives
\[
\int_0^1\!\!\int_0^1L_b\,c_b\de u\de v=\frac14+\frac{b|b|}8,
\]
hence \(\tau(L_b)=b|b|/2\).
\end{proof}

\begin{remark}\label{rem:mechanism}
The representation \(L_b=\Pi+\ell\otimes g_b\) explains the extremal structure.
The factor \(\ell(u)=\min\{u,1-u\}\) appears because the extremal kernel is constant on each half of the first coordinate.
The tent \(g_b\) is the pointwise smallest \(1\)-Lipschitz function in absolute value with \(g_b(\tfrac12)=b/2\).
And the cubic exponent in \(\xi=|\beta|^3/2\) comes from scaling, because the squared \(L^2\)-norm of a tent whose height and half-width are both proportional to \(|\beta|\) scales like \(|\beta|^3\).
\end{remark}

\section{The region for all copulas}\label{sec:proof-main}

The proof of Theorem~\ref{thm:main} has three steps.
Proposition~\ref{prop:sharp} turns the variational argument from the introduction into a proof.
It shows that no copula with \(\beta=b\) has a smaller value of \(\xi\) than the tent copula \(L_b\), and that \(L_b\) is the only copula attaining this minimum.
Proposition~\ref{prop:right-boundary} then constructs, for every value of \(\beta\), a deterministic copula with \(\xi=1\).
Lemma~\ref{lem:interpolation} fills in everything between the two boundaries.
The section closes with Corollary~\ref{cor:rs-pqd}, which reads off the regions for radially symmetric and quadrant dependent copulas from the same constructions.

\begin{proposition}[Sharp lower bound]\label{prop:sharp}
Every \(C\in\CC\) satisfies \(\xi(C)\ge|\beta(C)|^3/2\), with equality if and only if \(C=L_{\beta(C)}\).
\end{proposition}

\begin{proof}
Let \(b\coloneqq\beta(C)\) and define
\begin{align*}
F^+(v)&\coloneqq2C\bigl(\tfrac12,v\bigr)=2\int_0^{1/2}\partial_1C(u,v)\de u,\\
F^-(v)&\coloneqq2v-F^+(v)=2\int_{1/2}^1\partial_1C(u,v)\de u,
\end{align*}
the conditional distribution functions of \(V\) given \(U\le\tfrac12\) and given \(U>\tfrac12\), respectively.
Both are continuous and nondecreasing, with \(F^\pm(0)=0\) and \(F^\pm(1)=1\).
Put \(g(v)\coloneqq F^+(v)-v\), so that \(F^\pm(v)=v\pm g(v)\).
Monotonicity of \(F^+\) and \(F^-\) gives \(|g(t)-g(s)|\le t-s\) for \(s<t\), so \(g\) is \(1\)-Lipschitz, and \(g(\tfrac12)=2C(\tfrac12,\tfrac12)-\tfrac12=b/2\) by~\eqref{eq:beta-def}.
For every \(v\), Jensen's inequality on each half of the \(u\)-axis gives
\begin{align*}
\int_0^1\bigl(\partial_1C(u,v)\bigr)^2\de u
&\ge\tfrac12F^+(v)^2+\tfrac12F^-(v)^2\\
&=v^2+g(v)^2,
\end{align*}
and integrating in \(v\) proves~\eqref{eq:key} by~\eqref{eq:xi-def}.
As \(g\) is \(1\)-Lipschitz with \(g(\tfrac12)=b/2\), we have \(|g(v)|\ge|g_b(v)|\) for all \(v\), with \(g_b\) as in~\eqref{eq:tent}.
Consequently,
\[
\xi(C)\ge6\int_0^1g(v)^2\de v\ge6\int_0^1g_b(v)^2\de v=\frac{|b|^3}2.
\]

Suppose now that \(\xi(C)=|b|^3/2\).
Then \(|g|=|g_b|\) a.e., hence everywhere by continuity.
On \((\alpha_b,1-\alpha_b)\), the continuous function \(g\) never vanishes and \(g(\tfrac12)=b/2\), so \(g\) carries the sign of \(b\) there, and thus \(g=g_b\).
Equality also holds in Jensen's inequality for a.e.\ \(v\).
For a.e.\ \(v\), the function \(u\mapsto\partial_1C(u,v)\) is therefore a.e.\ constant on \((0,\tfrac12)\) and on \((\tfrac12,1)\), with values \(v\pm g_b(v)\).
By~\eqref{eq:Lb-kernel} and~\eqref{eq:kernel}, \(C(u,v)=L_b(u,v)\) for all \(u\) and a.e.\ \(v\), hence everywhere by continuity.
Conversely, \(\xi(L_b)=|b|^3/2\) by Proposition~\ref{prop:Lb}.
\end{proof}

The maximal value \(\xi=1\) turns out to be compatible with every value of \(\beta\).
For \(b\in[-1,1]\), let \(s_b\coloneqq(1+b)/4\in[0,\tfrac12]\), and let \(T_b\colon[0,1]\to[0,1]\) be the interval exchange
\[
T_b(u)\coloneqq
\begin{cases}
u+\tfrac12-s_b,&s_b<u\le\tfrac12,\\
u-\tfrac12+s_b,&\tfrac12<u\le1-s_b,\\
u,&\text{otherwise},
\end{cases}
\]
which swaps the two middle intervals of equal length \(\tfrac12-s_b\) and fixes the two outer ones.
Write \(D_b\) for the copula of \((U,T_b(U))\) with \(U\sim\mathrm{U}(0,1)\).
The top row of Figure~\ref{fig:singular} shows its support.

\begin{proposition}[Right boundary]\label{prop:right-boundary}
For every \(b\in[-1,1]\), the copula \(D_b\) is exchangeable and radially symmetric, and it satisfies \(\xi(D_b)=1\) and \(\beta(D_b)=b\).
For \(b\ge0\), \(D_b\) is \(\PQD\).
\end{proposition}

\begin{proof}
The map \(T_b\) is measure preserving.
Hence \(T_b(U)\) is uniform, \(D_b\) is completely dependent, and \(\xi(D_b)=1\) by Section~\ref{sec:preliminaries}.
Among the points \(u\le\tfrac12\), precisely those in \([0,s_b]\) are mapped into \([0,\tfrac12]\), so \(D_b(\tfrac12,\tfrac12)=s_b\) and \(\beta(D_b)=4s_b-1=b\).
Outside finitely many points, \(T_b(1-u)=1-T_b(u)\) and \(T_b(T_b(u))=u\).
Hence \((1-U,1-T_b(U))=(1-U,T_b(1-U))\) and \((T_b(U),U)=(T_b(U),T_b(T_b(U)))\) almost surely, and both vectors have the same law as \((U,T_b(U))\), because \(1-U\) and \(T_b(U)\) are uniform.
Thus \(D_b\) is radially symmetric and exchangeable.

Let now \(b\ge0\), so that \(s_b\ge\tfrac14\).
By radial symmetry, \(D_b(u,v)-uv=D_b(1-u,1-v)-(1-u)(1-v)\), so it suffices to show \(D_b(u,v)\ge uv\) for \(v\le\tfrac12\).
If \(v\le s_b\), then \(D_b(u,v)=\min\{u,v\}\ge uv\).
If \(s_b<v\le\tfrac12\), then \(T_b(U)\le v\) if and only if \(U\le s_b\) or \(\tfrac12<U\le v+\tfrac12-s_b\), so
\[
D_b(u,v)=\min\{u,s_b\}+\bigl(\min\{u,v+\tfrac12-s_b\}-\tfrac12\bigr)_+.
\]
For \(u\le s_b\) this equals \(u\ge uv\), and for \(s_b<u\le\tfrac12\) it equals \(s_b\ge\tfrac14\ge uv\).
For \(\tfrac12<u\le v+\tfrac12-s_b\) it equals \(s_b+u-\tfrac12\), and
\[
s_b+u-\tfrac12-uv\ge s_b-\tfrac v2\ge0,
\]
since \(u(1-v)\ge\tfrac12(1-v)\).
Finally, for \(u>v+\tfrac12-s_b\) it equals \(v\ge uv\).
\end{proof}

\begin{proof}[Proof of Theorem~\ref{thm:main}]
Since \(\xi\in[0,1]\) and \(\beta\in[-1,1]\), Proposition~\ref{prop:sharp} shows that \(\RR_{\xi,\beta}\) is contained in the set on the right-hand side, and it characterizes the equality case in~\eqref{eq:main-ineq}.
Conversely, fix \(b\in[-1,1]\).
By Propositions~\ref{prop:Lb} and~\ref{prop:right-boundary}, the copulas \(L_b\) and \(D_b\) attain the points \((|b|^3/2,b)\) and \((1,b)\), and Lemma~\ref{lem:interpolation} with \(\AAA=\CC\) yields all points \((x,b)\) with \(|b|^3/2\le x\le1\).
\end{proof}

The same constructions settle the region for radially symmetric and for quadrant dependent copulas, written \(\RR^{\RS}_{\xi,\beta}\), \(\RR^{\PQD}_{\xi,\beta}\), \(\RR^{\PQD,\RS}_{\xi,\beta}\), and so on.

\begin{corollary}[Radial symmetry and quadrant dependence]\label{cor:rs-pqd}
It holds that \(\RR^{\RS}_{\xi,\beta}=\RR_{\xi,\beta}\) and
\begin{align*}
\RR^{\PQD}_{\xi,\beta}=\RR^{\PQD,\RS}_{\xi,\beta}&=\RR_{\xi,\beta}\cap\bigl([0,1]\times[0,1]\bigr),\\
\RR^{\NQD}_{\xi,\beta}=\RR^{\NQD,\RS}_{\xi,\beta}&=\RR_{\xi,\beta}\cap\bigl([0,1]\times[-1,0]\bigr).
\end{align*}
\end{corollary}

\begin{proof}
By Propositions~\ref{prop:Lb-properties}~(i) and~\ref{prop:right-boundary}, \(L_b\) and \(D_b\) are radially symmetric, so Lemma~\ref{lem:interpolation} with \(\AAA=\CC_{\RS}\) and Theorem~\ref{thm:main} give \(\RR^{\RS}_{\xi,\beta}=\RR_{\xi,\beta}\).
If \(C\) is \(\PQD\), then \(C(\tfrac12,\tfrac12)\ge\tfrac14\), that is, \(\beta(C)\ge0\), which proves the inclusion \(\RR^{\PQD}_{\xi,\beta}\subseteq\RR_{\xi,\beta}\cap[0,1]^2\).
Conversely, for \(b\in[0,1]\), the copulas \(L_b\) and \(D_b\) belong to \(\CC_{\PQD}\cap\CC_{\RS}\) by Propositions~\ref{prop:Lb-properties}~(i),~(ii) and~\ref{prop:right-boundary}, and Lemma~\ref{lem:interpolation} with \(\AAA=\CC_{\PQD}\cap\CC_{\RS}\) yields every point \((x,b)\) with \(b^3/2\le x\le1\).
The \(\NQD\) statements follow from~\eqref{eq:reflect-region}.
\end{proof}

\section{The region for stochastically increasing copulas}\label{sec:si}

This section proves Theorem~\ref{thm:si}.
By~\eqref{eq:reflect-region}, it suffices to treat \(\SI\) copulas.
They are \(\PQD\), so they satisfy \(\beta\ge0\), and by Propositions~\ref{prop:sharp} and~\ref{prop:Lb-properties}~(ii) the lower bound \(\xi\ge\beta^3/2\) stays sharp within \(\CC_{\SI}\).
What is new is the upper bound~\eqref{eq:si-ineq}, and we prove it in two steps.
Lemma~\ref{lem:completion} replaces a copula by the largest copula with the same median section, which is again \(\SI\) and has a value of \(\xi\) at least as large.
The distance between the two atoms of this largest copula then turns the problem into a one-dimensional inequality for convex functions, Lemma~\ref{lem:convex}.
Its extremal functions lead to the two-branch copulas \(R_b\) of Proposition~\ref{prop:Rb}, which attain the new bound, and the proof of Theorem~\ref{thm:si} at the end of the section assembles the pieces.

\emph{Median sections.}
Let \(\mathcal S\) be the set of functions \(A(u)=\int_0^up(t)\de t\), \(u\in[0,1]\), where \(p\colon[0,1]\to[0,1]\) is nonincreasing with \(\int_0^1p(t)\de t=\tfrac12\).
Each \(A\in\mathcal S\) is concave and nondecreasing with \(A(0)=0\) and \(A(1)=\tfrac12\), and we put
\[
B(u)\coloneqq\tfrac12+u-A(u),\qquad u\in[0,1].
\]
Since \(0\le p\le1\), the function \(B\) is nondecreasing with \(B(0)=\tfrac12\) and \(B(1)=1\), and \(\max\{0,u-\tfrac12\}\le A(u)\le\min\{u,\tfrac12\}\), hence \(A(u)\le\tfrac12\le B(u)\).
For \(A\in\mathcal S\), define
\begin{equation}\label{eq:QA}
Q_A(u,v)\coloneqq
\begin{cases}
\min\{A(u),v\},&v\le\tfrac12,\\
\min\{u,A(u)+v-\tfrac12\},&v\ge\tfrac12.
\end{cases}
\end{equation}
If \(C\in\CC_{\SI}\) and the kernel is chosen as in the definition of stochastic increasingness, then \(p\coloneqq\partial_1C(\cdot,\tfrac12)\) is nonincreasing, and by~\eqref{eq:kernel} the median section \(A\coloneqq C(\cdot,\tfrac12)\) belongs to \(\mathcal S\), with \(A(\tfrac12)=(1+\beta(C))/4\).

\begin{lemma}[Maximal completion of the median section]\label{lem:completion}
Let \(A\in\mathcal S\).
\begin{enumerate}[label=\textup{(\roman*)},leftmargin=*]
\item The function \(Q_A\) in~\eqref{eq:QA} is an \(\SI\) copula with median section \(Q_A(\cdot,\tfrac12)=A\).
Given \(U=u\), its conditional law of \(V\) puts mass \(p(u)\) on \(A(u)\) and mass \(1-p(u)\) on \(B(u)\).
\item Every \(C\in\CC\) with \(C(\cdot,\tfrac12)=A\) satisfies \(C\le Q_A\).
\item If \(C\in\CC_{\SI}\) satisfies \(C(\cdot,\tfrac12)=A\), then \(\xi(C)\le\xi(Q_A)\), with equality if and only if \(C=Q_A\).
\end{enumerate}
\end{lemma}

\begin{proof}
(i) Let \(K(u,\cdot)\coloneqq p(u)\delta_{A(u)}+(1-p(u))\delta_{B(u)}\).
Since \(A(u)\le\tfrac12\le B(u)\), its distribution function \(k_v(u)\coloneqq K(u,[0,v])\) is
\[
k_v(u)=
\begin{cases}
p(u)\1_{\{A(u)\le v\}},&v<\tfrac12,\\
p(u)+(1-p(u))\1_{\{B(u)\le v\}},&v\ge\tfrac12.
\end{cases}
\]
As \(A\) and \(B\) are nondecreasing, the sets \(\{u:A(u)\le v\}\) and \(\{u:B(u)\le v\}\) are initial segments of \([0,1]\).
For fixed \(v<\tfrac12\), the function \(k_v\) thus equals the nonincreasing function \(p\) on an initial segment and zero afterwards.
For fixed \(v\ge\tfrac12\), it equals one on an initial segment and \(p\le1\) afterwards.
In both cases \(k_v\) is nonincreasing.
Since \(A\) and \(B\) are continuous and nondecreasing with derivatives \(p\) and \(1-p\) a.e., \(A(0)=0\), \(A(1)=\tfrac12\), \(B(0)=\tfrac12\), and \(B(1)=1\), integration gives
\[
\int_0^uk_v(t)\de t=\min\{A(u),v\}
\]
for \(v<\tfrac12\), and for \(v\ge\tfrac12\)
\begin{align*}
\int_0^uk_v(t)\de t&=A(u)+\min\{B(u),v\}-\tfrac12\\
&=\min\{u,A(u)+v-\tfrac12\}.
\end{align*}
Thus \(\int_0^uk_v(t)\de t=Q_A(u,v)\), and \(Q_A(1,v)=v\) for all \(v\).
By Section~\ref{sec:preliminaries}, \(Q_A\) is a copula with kernel \(K\), it is \(\SI\), and \(Q_A(\cdot,\tfrac12)=A\).

(ii) For \(v\le\tfrac12\), we have \(C(u,v)\le C(u,\tfrac12)=A(u)\) and \(C(u,v)\le v\).
For \(v\ge\tfrac12\), we have \(C(u,v)\le u\) and, by the Lipschitz property, \(C(u,v)\le A(u)+v-\tfrac12\).
Hence \(C\le Q_A\).

(iii) Choose the kernel of \(C\) as in the definition of stochastic increasingness.
Fix \(v\), write \(h(u)\coloneqq\partial_1C(u,v)\), which is nonincreasing in \(u\), and \(k\coloneqq k_v\), and let
\[
G(u)\coloneqq\int_0^u\bigl(k(t)-h(t)\bigr)\de t=Q_A(u,v)-C(u,v),
\]
which is nonnegative by~(ii).
Since \(h\) is nonincreasing with values in \([0,1]\), the set \(\{u:h(u)>t\}\) is, up to at most one point, an interval \([0,u_t)\) for every \(t\in[0,1)\).
Writing \(h(u)=\int_0^1\1_{\{h(u)>t\}}\de t\), Fubini's theorem yields
\[
\int_0^1h\,(k-h)\de u=\int_0^1G(u_t)\de t\ge0.
\]
Consequently,
\[
\int_0^1(k^2-h^2)\de u=\int_0^1(k-h)^2\de u+2\int_0^1h\,(k-h)\de u,
\]
and the right-hand side is at least \(\int_0^1(k-h)^2\de u\).
Integrating in \(v\) and using~\eqref{eq:xi-def} gives
\begin{align*}
&\xi(Q_A)-\xi(C)\\
&\quad\ge6\int_0^1\!\!\int_0^1\bigl(k_v(u)-\partial_1C(u,v)\bigr)^2\de u\de v\ge0.
\end{align*}
If \(\xi(C)=\xi(Q_A)\), then the kernels of \(C\) and \(Q_A\) coincide for a.e.\ \((u,v)\), so \(C=Q_A\) by~\eqref{eq:kernel} and continuity.
\end{proof}

\emph{The gap function.}
For \(A\in\mathcal S\), the distance between the two atoms in Lemma~\ref{lem:completion}~(i) is
\begin{equation}\label{eq:gap}
w(u)\coloneqq B(u)-A(u)=\tfrac12+u-2A(u).
\end{equation}
The function \(w\) is convex, satisfies \(w(0)=w(1)=\tfrac12\), \(w(u)\ge|u-\tfrac12|\) because \(A(u)\le\min\{u,\tfrac12\}\), and \(w'=1-2p\) a.e.
Given \(U=u\), the conditional distribution function of \(V\) under \(Q_A\) equals \(p(u)\) on \([A(u),B(u))\) and lies in \(\{0,1\}\) elsewhere.
Hence~\eqref{eq:xi-deficit} and \(p(1-p)=\tfrac14(1-w'^2)\) give
\begin{equation}\label{eq:J}
1-\xi(Q_A)=\tfrac32J(w),
\end{equation}
where \(J(w)\coloneqq\int_0^1w(u)\bigl(1-w'(u)^2\bigr)\de u\).
Finally, \(w(\tfrac12)=1-2A(\tfrac12)\), so that \(w(\tfrac12)=(1-\beta(C))/2\) if \(A\) is the median section of \(C\).

\begin{lemma}[A sharp inequality for convex functions]\label{lem:convex}
Let \(w\colon[0,1]\to\R\) be convex with \(w(0)=w(1)=\tfrac12\) and \(w(u)\ge|u-\tfrac12|\) for all \(u\), and let \(q\coloneqq w(\tfrac12)\).
Then \(w\) is \(1\)-Lipschitz, \(0\le q\le\tfrac12\), and
\[
J(w)\ge2q^2,
\]
with equality if and only if \(w(u)=\max\{|u-\tfrac12|,q\}\) for all \(u\).
\end{lemma}

\begin{proof}
The function \(w\) is convex, lies above \(|u-\tfrac12|\), and agrees with it at \(0\) and \(1\), so its one-sided derivatives lie in \([-1,1]\) and \(w\) is \(1\)-Lipschitz.
Convexity and \(w(0)=w(1)=\tfrac12\) give \(q\le\tfrac12\), and \(q\ge0\) holds by assumption.
If \(q=0\), then \(w\) lies below the chords from \((0,\tfrac12)\) to \((\tfrac12,0)\) and from \((\tfrac12,0)\) to \((1,\tfrac12)\), which form the graph of \(|u-\tfrac12|\).
Hence \(w(u)=|u-\tfrac12|\), \(|w'|=1\) a.e., and \(J(w)=0\).

Let \(q>0\) and choose a subgradient \(m\) of \(w\) at \(\tfrac12\).
Evaluating the supporting line \(u\mapsto q+m(u-\tfrac12)\) at \(0\) and at \(1\) gives \(|m|\le1-2q<1\).
By monotonicity of \(w'\), we have \(w'\le m\) a.e.\ on \((0,\tfrac12)\) and \(w'\ge m\) a.e.\ on \((\tfrac12,1)\).
Since \(|w'|\le1\), this yields
\[
1-w'^2\ge
\begin{cases}
(1-m)(1+w'),&u<\tfrac12,\\
(1+m)(1-w'),&u>\tfrac12.
\end{cases}
\]
Multiplying by \(w\ge0\), integrating, and using
\[
\int_0^{1/2}w\,w'\de u=-\int_{1/2}^1w\,w'\de u=\tfrac{q^2}2-\tfrac18,
\]
we obtain
\begin{align*}
J(w)\ge{}&(1-m)\bigl(I_-+\tfrac{q^2}2-\tfrac18\bigr)\\
&+(1+m)\bigl(I_++\tfrac{q^2}2-\tfrac18\bigr),
\end{align*}
where \(I_-\coloneqq\int_0^{1/2}w\de u\) and \(I_+\coloneqq\int_{1/2}^1w\de u\).
By convexity and assumption,
\[
w(u)\ge\max\bigl\{|u-\tfrac12|,\,q+m(u-\tfrac12)\bigr\}.
\]
The two lines in the maximum intersect at distance \(q/(1+m)\le\tfrac12\) to the left and \(q/(1-m)\le\tfrac12\) to the right of \(\tfrac12\), so the areas of the two triangles between them give
\begin{equation}\label{eq:area}
I_-\ge\frac18+\frac{q^2}{2(1+m)},\qquad I_+\ge\frac18+\frac{q^2}{2(1-m)}.
\end{equation}
Inserting~\eqref{eq:area} yields
\begin{align*}
J(w)&\ge\frac{q^2}2\Bigl[(1-m)\frac{2+m}{1+m}+(1+m)\frac{2-m}{1-m}\Bigr]\\
&=\frac{2q^2}{1-m^2}\ge2q^2.
\end{align*}
If \(J(w)=2q^2\), then \(m=0\), and equality holds in both inequalities in~\eqref{eq:area}.
Since \(w-\max\{|u-\tfrac12|,q\}\) is continuous and nonnegative, this forces \(w=\max\{|u-\tfrac12|,q\}\).
Conversely, this function has derivative zero on the interval \((\tfrac12-q,\tfrac12+q)\) of length \(2q\), where it equals \(q\), and derivative \(\pm1\) elsewhere, so \(J(w)=2q^2\).
\end{proof}

\emph{The right-boundary copulas.}
For \(b\in[0,1]\), define
\begin{equation}\label{eq:Ab}
A_b(u)\coloneqq\min\bigl\{u,\tfrac14(2u+b),\tfrac12\bigr\},\qquad R_b\coloneqq Q_{A_b}.
\end{equation}
Then \(A_b\in\mathcal S\), with \(p_b=1\) on \([0,\tfrac b2)\), \(p_b=\tfrac12\) on \([\tfrac b2,1-\tfrac b2)\), and \(p_b=0\) on \([1-\tfrac b2,1]\).
Its gap function~\eqref{eq:gap} is \(w_b(u)=\max\{|u-\tfrac12|,(1-b)/2\}\), the extremal function of Lemma~\ref{lem:convex} with \(q=(1-b)/2\).
The bottom row of Figure~\ref{fig:singular} displays the supports of \(R_b\).

\begin{figure*}[t!]
\centering
\includegraphics[width=\textwidth]{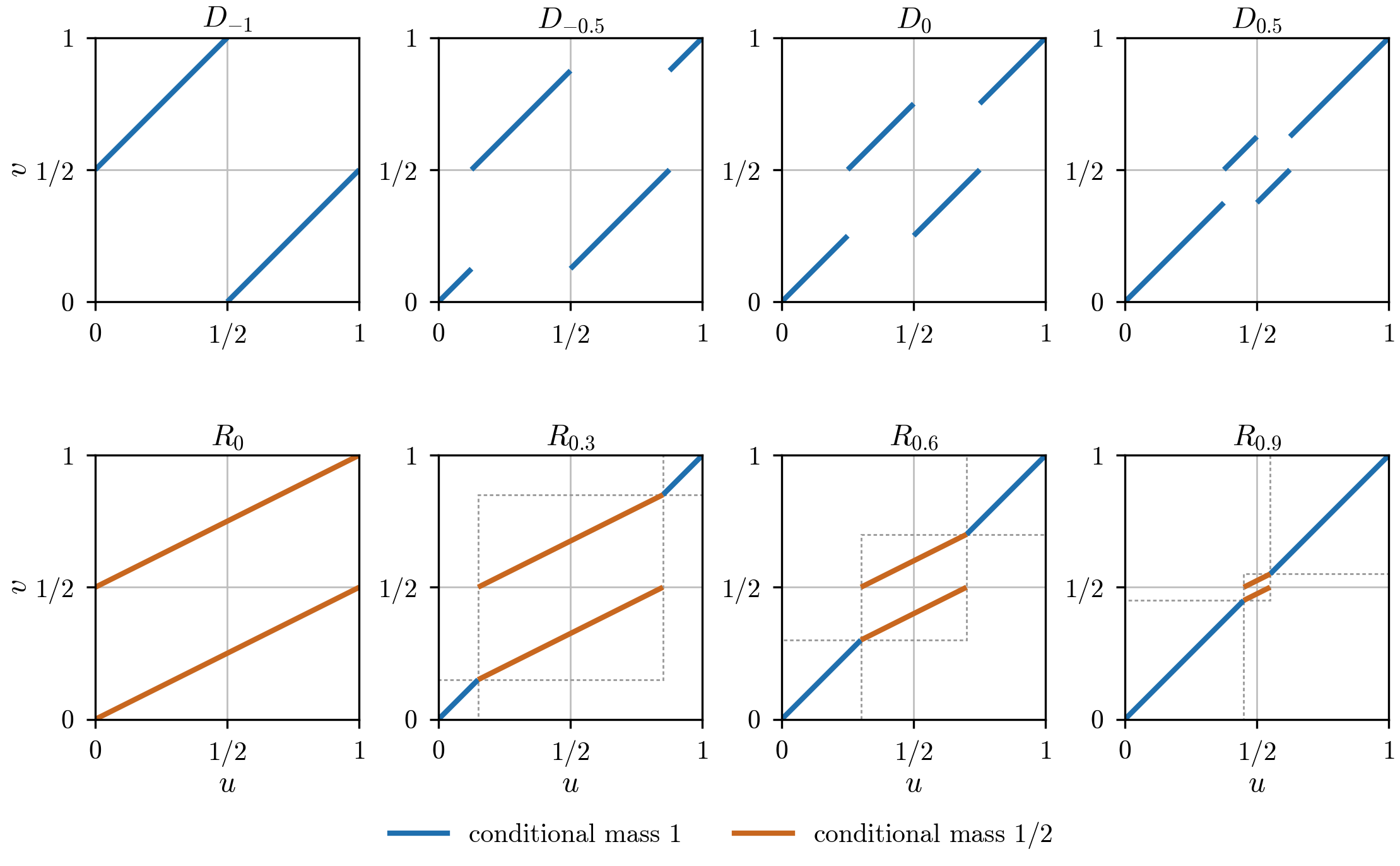}
\caption{Supports of the singular boundary copulas.
Top row: the deterministic copulas \(D_b\) of Proposition~\ref{prop:right-boundary} for \(b\in\{-1,-0.5,0,0.5\}\), which attain the right boundary \(\xi=1\) of \(\RR_{\xi,\beta}\).
Given \(U=u\), the second coordinate equals \(T_b(u)\).
Bottom row: the two-branch copulas \(R_b\) of Proposition~\ref{prop:Rb} for \(b\in\{0,0.3,0.6,0.9\}\), which attain the right boundary \(\xi=1-\tfrac34(1-\beta)^2\) of \(\RR^{\SI}_{\xi,\beta}\).
On the blue diagonal segments \(V=U\).
On the middle block, \(V\) lies on each of the two orange branches of slope \(\tfrac12\) with conditional probability \(\tfrac12\).
Dashed squares indicate the ordinal sum structure of \(R_b\).
For \(b=1\), both families collapse to \(D_1=R_1=M\).}
\label{fig:singular}
\end{figure*}

\begin{proposition}[Two-branch copulas]\label{prop:Rb}
Let \(b\in[0,1]\), let \(U\sim\mathrm{U}(0,1)\) and \(\varepsilon\sim\mathrm{Bernoulli}(\tfrac12)\) be independent, and define \(V_b\coloneqq U\) if \(U\notin(\tfrac b2,1-\tfrac b2)\) and
\[
V_b\coloneqq\tfrac14(2U+b)+\tfrac12\varepsilon(1-b)
\]
otherwise.
Then the copula \(R_b\) in~\eqref{eq:Ab} is the copula of \((U,V_b)\), and it has the following properties.
\begin{enumerate}[label=\textup{(\roman*)},leftmargin=*]
\item \(R_b\) is singular, \(\SI\), and radially symmetric.
For \(b<1\), neither is \(R_b\) exchangeable nor is \(R_b^\top\) \(\SI\), whereas \(R_1=M\).
\item For \(0<b<1\), \(R_b\) is the ordinal sum of \(M\), \(R_0\), and \(M\) with respect to the partition \(\{[0,\tfrac b2],[\tfrac b2,1-\tfrac b2],[1-\tfrac b2,1]\}\).
\item \(\beta(R_b)=b\) and \(\xi(R_b)=1-\tfrac34(1-b)^2\).
\item \(\tau(R_b)=1-\tfrac12(1-b)^2\) and \(\rho(R_b)=1-\tfrac12(1-b)^3\).
\end{enumerate}
\end{proposition}

\begin{proof}
Let \(B_b(u)\coloneqq\tfrac12+u-A_b(u)\).
By Lemma~\ref{lem:completion}~(i), given \(U=u\), the second coordinate under \(R_b\) equals \(A_b(u)=u\) for \(u<\tfrac b2\), and \(B_b(u)=u\) for \(u\ge1-\tfrac b2\).
For \(\tfrac b2\le u<1-\tfrac b2\), it equals \(A_b(u)=\tfrac14(2u+b)\) or \(B_b(u)=\tfrac14(2u+b)+\tfrac12(1-b)\) with probability \(\tfrac12\) each.
This is the law of \((U,V_b)\).

(i) The support of \(R_b\) is a finite union of line segments, so \(R_b\) is singular, and it is \(\SI\) by Lemma~\ref{lem:completion}~(i).
A direct computation shows \(1-V_b=\tfrac14(2(1-U)+b)+\tfrac12(1-\varepsilon)(1-b)\) if \(U\in(\tfrac b2,1-\tfrac b2)\), and \(1-V_b=1-U\) otherwise.
Since \((1-U,1-\varepsilon)\) has the same law as \((U,\varepsilon)\), the vector \((1-U,1-V_b)\) has the same law as \((U,V_b)\), so \(R_b\) is radially symmetric.
For \(b=1\), we have \(V_1=U\) and \(R_1=M\).
Let \(b<1\).
The point \(v_0\coloneqq A_b(\tfrac12)=(1+b)/4\) lies in \([\tfrac b2,\tfrac12]\), and \(R_b\) is not exchangeable, since
\[
R_b\bigl(\tfrac12,v_0\bigr)=v_0>\tfrac18(1+3b)=A_b(v_0)=R_b\bigl(v_0,\tfrac12\bigr).
\]
The transpose \(R_b^\top\) is \(\SI\) if and only if \(v\mapsto R_b(u,v)\) is concave for every \(u\).
For \(u\in(\tfrac b2,1-\tfrac b2)\), we have \(A_b(u)<\tfrac12<B_b(u)\), and~\eqref{eq:QA} shows that \(v\mapsto R_b(u,v)\) has slope \(0\) on \((A_b(u),\tfrac12)\) and slope \(1\) on \((\tfrac12,B_b(u))\), so it is not concave.

(ii) For \(U\in(\tfrac b2,1-\tfrac b2)\), the affine rescalings \(U'\coloneqq(U-\tfrac b2)/(1-b)\) and \(V'\coloneqq(V_b-\tfrac b2)/(1-b)\) satisfy \(V'=\tfrac12(U'+\varepsilon)\), which is the representation of \(R_0\).
Together with \(V_b=U\) outside this interval, this is the ordinal sum structure, see \cite[Sec.~3.2]{nelsen2006introduction}.

(iii) We have \(\beta(R_b)=4A_b(\tfrac12)-1=b\).
Since the gap function of \(A_b\) is \(w_b\), Lemma~\ref{lem:convex} and~\eqref{eq:J} give \(1-\xi(R_b)=\tfrac32\cdot2q^2\) with \(q=(1-b)/2\).

(iv) On the support, \(R_b(u,u)=u\) for \(u\notin(\tfrac b2,1-\tfrac b2)\), and \(R_b(u,A_b(u))=A_b(u)\) and \(R_b(u,B_b(u))=u\) otherwise.
Hence \(\tau(R_b)=4\,\mathbb E[R_b(U,V_b)]-1\) and \(\rho(R_b)=12\,\mathbb E[UV_b]-3\) reduce to integrals of polynomials over the three intervals, and evaluating them yields the stated values.
\end{proof}

\begin{proof}[Proof of Theorem~\ref{thm:si}]
Let \(C\in\CC_{\SI}\) and \(b\coloneqq\beta(C)\).
Since \(C\) is \(\PQD\), we have \(b\ge0\), and Proposition~\ref{prop:sharp} gives \(b^3\le2\xi(C)\).
Let \(A\coloneqq C(\cdot,\tfrac12)\in\mathcal S\) and let \(w\) be its gap function~\eqref{eq:gap}, so that \(w(\tfrac12)=(1-b)/2\).
By Lemma~\ref{lem:completion}~(iii),~\eqref{eq:J}, and Lemma~\ref{lem:convex},
\[
1-\xi(C)\ge1-\xi(Q_A)=\tfrac32J(w)\ge\tfrac34(1-b)^2,
\]
which is~\eqref{eq:si-ineq}, or equivalently \(3(1-b)^2\le4(1-\xi(C))\).
If equality holds, then \(C=Q_A\) by Lemma~\ref{lem:completion}~(iii), and \(w=w_b\) by Lemma~\ref{lem:convex}.
Since \(A\) is determined by \(w\) through~\eqref{eq:gap}, we get \(A=A_b\) and thus \(C=R_b\).
Conversely, \(R_b\) attains equality by Proposition~\ref{prop:Rb}~(iii).
This proves that \(\RR^{\SI}_{\xi,\beta}\) is contained in the stated set and characterizes the equality case in~\eqref{eq:si-ineq}.

For the reverse inclusion, fix \(b\in[0,1]\).
By Propositions~\ref{prop:Lb},~\ref{prop:Lb-properties}~(i),~(ii), and~\ref{prop:Rb}, the copulas \(L_b\) and \(R_b\) belong to the convex class \(\CC_{\SI}\cap\CC_{\RS}\), both satisfy \(\beta=b\), and they attain \(\xi=b^3/2\) and \(\xi=1-\tfrac34(1-b)^2\), respectively.
Lemma~\ref{lem:interpolation} yields all intermediate values, so \(\RR^{\SI,\RS}_{\xi,\beta}\), and hence \(\RR^{\SI}_{\xi,\beta}\), contains the stated set.
The statement for \(\SD\) copulas follows from~\eqref{eq:reflect-region}.
\end{proof}

\begin{remark}\label{rem:R0}
\begin{enumerate}[label=(\alph*),leftmargin=*]
\item At \(\beta=0\), the \(\SI\) region is the segment \(0\le\xi\le\tfrac14\), and its upper endpoint \(R_0\) is the copula of \((U,\tfrac12(U+\varepsilon))\).
The transpose of \(R_0\) is completely dependent, since \(U=2V_0-\varepsilon\) and \(\varepsilon=\lfloor2V_0\rfloor\) almost surely.
Hence \(\xi(R_0^\top)=1\) while \(\xi(R_0)=\tfrac14\), and \(R_0\) is \(\SI\) while \(R_0^\top\) is not.
Both \(\xi\) and stochastic increasingness are directional, and \(R_0\) makes this visible.
\item By Proposition~\ref{prop:Rb}, \(\xi(R_b)\le\tau(R_b)\le\rho(R_b)\) for all \(b\in[0,1]\), in accordance with the inequality \(\xi\le\tau\) for \(\SI\) copulas in \cite[Thm.~1.1]{rockel2026kendall}.
\end{enumerate}
\end{remark}

\section{Concluding remarks}\label{sec:conclusion}

The exact region between Chatterjee's rank correlation and Blomqvist's beta is the cubic epigraph \(\xi\ge|\beta|^3/2\).
A one-dimensional variational problem governs its left boundary.
The value of \(\beta\) fixes the displacement of a conditional distribution function at the median, monotonicity lets this displacement decay with slope at most one, and the unique least-energy profile is the tent \(g_b\).
Under stochastic increasingness, the median section of the copula becomes concave, the gap function \(w\) becomes convex, and maximizing \(\xi\) reduces to the sharp inequality of Lemma~\ref{lem:convex}, whose extremal functions produce the two-branch copulas \(R_b\).

Several related questions remain open.
For \(0<|b|<1\), the tent copulas \(L_b\) are not exchangeable, and for \(0<b<1\), neither \(L_b^\top\) nor \(R_b^\top\) is stochastically increasing.
The present constructions therefore determine neither the left boundary of the region for exchangeable copulas nor either boundary of the region for copulas that are stochastically increasing in both directions.
With the exact regions of \(\xi\) with \(\rho\) \cite{ansari2026exact}, with the footrule \cite{rockel2026exact}, and now with \(\beta\) in hand, it is also natural to ask for three-dimensional regions such as that of \((\xi,\beta,\rho)\), in analogy with the recent results for triples of concordance measures \cite{bukovvsek2025exact,kokolbukovsek2026footrule,orenday2026taufootrulebeta}.

\bibliographystyle{plainnat}
\bibliography{cas-refs}

\begin{thebibliography}{30}
\providecommand{\natexlab}[1]{#1}
\providecommand{\url}[1]{\texttt{#1}}
\expandafter\ifx\csname urlstyle\endcsname\relax
  \providecommand{\doi}[1]{doi: #1}\else
  \providecommand{\doi}{doi: \begingroup \urlstyle{rm}\Url}\fi

\bibitem[Ansari and Rockel(2026)]{ansari2026exact}
Jonathan Ansari and Marcus Rockel.
\newblock The exact region and an inequality between {Chatterjee}'s and
  {Spearman}'s rank correlations.
\newblock \emph{J. Multivariate Anal.}, 214:\penalty0 105630, 2026.

\bibitem[Ansari et~al.(2026)Ansari, Langthaler, Fuchs, and
  Trutschnig]{Ansari-LFT-2025}
Jonathan Ansari, Patrick~B. Langthaler, Sebastian Fuchs, and Wolfgang
  Trutschnig.
\newblock Quantifying and estimating dependence via sensitivity of conditional
  distributions.
\newblock \emph{Bernoulli}, 32\penalty0 (1), 2026.

\bibitem[Auddy et~al.(2024)Auddy, Deb, and Nandy]{auddy2021exact}
Arnab Auddy, Nabarun Deb, and Sagnik Nandy.
\newblock Exact detection thresholds and minimax optimality of {Chatterjee}'s
  correlation coefficient.
\newblock \emph{Bernoulli}, 30\penalty0 (2):\penalty0 1640--1668, 2024.

\bibitem[Azadkia and Chatterjee(2021)]{chatterjee2021}
Mona Azadkia and Sourav Chatterjee.
\newblock A simple measure of conditional dependence.
\newblock \emph{Ann. Statist.}, 49\penalty0 (6):\penalty0 3070--3102, 2021.

\bibitem[Blomqvist(1950)]{blomqvist1950measure}
Nils Blomqvist.
\newblock On a measure of dependence between two random variables.
\newblock \emph{Ann. Math. Stat.}, 21\penalty0 (4):\penalty0 593--600, 1950.

\bibitem[Chatterjee(2021)]{chatterjee2020}
Sourav Chatterjee.
\newblock A new coefficient of correlation.
\newblock \emph{J. Am. Statist. Assoc.}, 116\penalty0 (536):\penalty0
  2009--2022, 2021.

\bibitem[Darsow et~al.(1992)Darsow, Nguyen, and Olsen]{darsow1992copulas}
William~F. Darsow, Bao Nguyen, and Elwood~T. Olsen.
\newblock Copulas and {Markov} processes.
\newblock \emph{Illinois J. Math.}, 36\penalty0 (4):\penalty0 600--642, 1992.

\bibitem[Dette et~al.(2013)Dette, Siburg, and Stoimenov]{dette2013copula}
Holger Dette, Karl~F. Siburg, and Pavel~A. Stoimenov.
\newblock A copula-based non-parametric measure of regression dependence.
\newblock \emph{Scand. J. Stat.}, 40\penalty0 (1):\penalty0 21--41, 2013.

\bibitem[{Durante} and {Sempi}(2016)]{durante2016principles}
Fabrizio {Durante} and Carlo {Sempi}.
\newblock \emph{{Principles of Copula Theory}}.
\newblock Boca Raton, FL: CRC Press, 2016.

\bibitem[Fuchs et~al.(2026)Fuchs, Limbach, and Sch{\"u}rrer]{fuchs2025exact}
Sebastian Fuchs, Carsten Limbach, and Fabian Sch{\"u}rrer.
\newblock On exact regions between measures of concordance and {Chatterjee}'s
  rank correlation for lower semilinear copulas.
\newblock \emph{Int. J. Approx. Reasoning}, 189:\penalty0 109588, 2026.

\bibitem[Huang et~al.(2022)Huang, Deb, and Sen]{deb2020b}
Zhen Huang, Nabarun Deb, and Bodhisattva Sen.
\newblock Kernel partial correlation coefficient -- a measure of conditional
  dependence.
\newblock \emph{J. Mach. Learn. Res.}, 23:\penalty0 216, 2022.

\bibitem[Karlin(1968)]{karlin1968total}
Samuel Karlin.
\newblock \emph{Total Positivity. {Vol}. {I}}.
\newblock Stanford University Press, Stanford, Calif., 1968.

\bibitem[Karlin and Rinott(1980{\natexlab{a}})]{karlinrinott1980classes}
Samuel Karlin and Yosef Rinott.
\newblock Classes of orderings of measures and related correlation
  inequalities. {I}: Multivariate totally positive distributions.
\newblock \emph{J. Multivariate Anal.}, 10\penalty0 (4):\penalty0 467--498,
  1980{\natexlab{a}}.

\bibitem[Karlin and Rinott(1980{\natexlab{b}})]{karlinrinott1980reverse}
Samuel Karlin and Yosef Rinott.
\newblock Classes of orderings of measures and related correlation
  inequalities. {II}: Multivariate reverse rule distributions.
\newblock \emph{J. Multivariate Anal.}, 10\penalty0 (4):\penalty0 499--516,
  1980{\natexlab{b}}.

\bibitem[Kokol~Bukov{\v{s}}ek and Moj{\v{s}}kerc(2022)]{bukovvsek2022exact}
Damjana Kokol~Bukov{\v{s}}ek and Bla{\v{z}} Moj{\v{s}}kerc.
\newblock On the exact region determined by {Spearman}'s footrule and {Gini}'s
  gamma.
\newblock \emph{J. Comput. Appl. Math.}, 410:\penalty0 114212, 2022.

\bibitem[Kokol~Bukov{\v{s}}ek and Moj{\v{s}}kerc(2026)]{bukovvsek2025exact}
Damjana Kokol~Bukov{\v{s}}ek and Bla{\v{z}} Moj{\v{s}}kerc.
\newblock The exact region determined by {Blomqvist}'s beta, {Spearman}'s
  footrule and {Gini}'s gamma.
\newblock \emph{J. Comput. Appl. Math.}, 473:\penalty0 116861, 2026.

\bibitem[Kokol~Bukov{\v{s}}ek and Stopar(2023)]{bukovvsek2023exact}
Damjana Kokol~Bukov{\v{s}}ek and Nik Stopar.
\newblock On the exact regions determined by {Kendall}'s tau and other
  concordance measures.
\newblock \emph{Mediterr. J. Math.}, 20\penalty0 (3):\penalty0 147, 2023.

\bibitem[Kokol~Bukov{\v{s}}ek and Stopar(2024)]{bukovvsek2024exact}
Damjana Kokol~Bukov{\v{s}}ek and Nik Stopar.
\newblock On the exact region determined by {Spearman}'s rho and {Spearman}'s
  footrule.
\newblock \emph{J. Comput. Appl. Math.}, 437:\penalty0 115463, 2024.

\bibitem[Kokol~Bukov{\v{s}}ek et~al.(2019)Kokol~Bukov{\v{s}}ek, Ko{\v{s}}ir,
  Moj{\v{s}}kerc, and Omladi{\v{c}}]{kokolbukovsek2019relation}
Damjana Kokol~Bukov{\v{s}}ek, Toma{\v{z}} Ko{\v{s}}ir, Bla{\v{z}}
  Moj{\v{s}}kerc, and Matja{\v{z}} Omladi{\v{c}}.
\newblock Relation between {Blomqvist}'s beta and other measures of concordance
  of copulas.
\newblock Preprint, arXiv:1911.03467, 2019.

\bibitem[Kokol~Bukov{\v{s}}ek et~al.(2021)Kokol~Bukov{\v{s}}ek, Ko{\v{s}}ir,
  Moj{\v{s}}kerc, and Omladi{\v{c}}]{bukovvsek2021exact}
Damjana Kokol~Bukov{\v{s}}ek, Toma{\v{z}} Ko{\v{s}}ir, Bla{\v{z}}
  Moj{\v{s}}kerc, and Matja{\v{z}} Omladi{\v{c}}.
\newblock Spearman's footrule and {Gini}'s gamma: local bounds for bivariate
  copulas and the exact region with respect to {Blomqvist}'s beta.
\newblock \emph{J. Comput. Appl. Math.}, 390:\penalty0 113385, 2021.

\bibitem[Kokol~Bukov{\v{s}}ek et~al.(2026)Kokol~Bukov{\v{s}}ek, Lazi{\'c},
  Moj{\v{s}}kerc, and Stopar]{kokolbukovsek2026footrule}
Damjana Kokol~Bukov{\v{s}}ek, Petra Lazi{\'c}, Bla{\v{z}} Moj{\v{s}}kerc, and
  Nik Stopar.
\newblock The exact region determined by {Spearman}'s footrule, {Gini}'s gamma
  and {Kendall}'s tau.
\newblock \emph{J. Comput. Appl. Math.}, 486:\penalty0 117639, 2026.

\bibitem[Lehmann(1966)]{lehmann1966concepts}
E.~L. Lehmann.
\newblock Some concepts of dependence.
\newblock \emph{Ann. Math. Stat.}, 37\penalty0 (5):\penalty0 1137--1153, 1966.

\bibitem[{Nelsen}(2006)]{nelsen2006introduction}
Roger~B. {Nelsen}.
\newblock \emph{{An Introduction to Copulas. 2nd ed.}}
\newblock New York, NY: Springer, 2006.

\bibitem[Orenday~Lares and Rockel(2026)]{orenday2026taufootrulebeta}
Jacob~Israel Orenday~Lares and Marcus Rockel.
\newblock The exact region determined by {Kendall}'s tau, {Spearman}'s footrule
  and {Blomqvist}'s beta.
\newblock Preprint, arXiv:2607.12841, 2026.

\bibitem[Rockel(2026{\natexlab{a}})]{rockel2026blest}
Marcus Rockel.
\newblock The exact region between {Chatterjee}'s and {Blest}'s rank
  correlations.
\newblock \emph{Int. J. Approx. Reasoning}, 197:\penalty0 109744,
  2026{\natexlab{a}}.

\bibitem[Rockel(2026{\natexlab{b}})]{rockel2026exact}
Marcus Rockel.
\newblock On the exact region between {Chatterjee}'s rank correlation and
  {Spearman}'s footrule.
\newblock \emph{J. Comput. Appl. Math.}, 485:\penalty0 117466,
  2026{\natexlab{b}}.

\bibitem[Rockel(2026{\natexlab{c}})]{rockel2026kendall}
Marcus Rockel.
\newblock Kendall and {Spearman} bounds for {Chatterjee}'s rank correlation
  under positive dependence.
\newblock Preprint, arXiv:2606.22074, 2026{\natexlab{c}}.

\bibitem[Schreyer et~al.(2017)Schreyer, Paulin, and
  Trutschnig]{schreyer2017exact}
Manuela Schreyer, Roland Paulin, and Wolfgang Trutschnig.
\newblock On the exact region determined by {Kendall}'s {{\(\tau\)}} and
  {Spearman}'s {{\(\rho\)}}.
\newblock \emph{J. R. Stat. Soc., Ser. B, Stat. Methodol.}, 79\penalty0
  (2):\penalty0 613--633, 2017.

\bibitem[Shi et~al.(2022)Shi, Drton, and Han]{shi2022power}
H.~Shi, M.~Drton, and F.~Han.
\newblock On the power of {Chatterjee}'s rank correlation.
\newblock \emph{Biometrika}, 109\penalty0 (2):\penalty0 317--333, 2022.

\bibitem[Sklar(1959)]{sklar1959fonctions}
Abe Sklar.
\newblock Fonctions de r{\'e}partition {\`a} {{\(n\)}} dimensions et leurs
  marges.
\newblock \emph{Publ. Inst. Stat. Univ. Paris}, 8:\penalty0 229--231, 1959.

\end{thebibliography}

\end{document}